\documentclass[a4paper,10pt]{amsart}
\usepackage{latexsym,amsfonts,amsmath, amsthm, amssymb}

\usepackage{graphics,graphicx}
\usepackage{a4wide}

\newtheorem*{thm}{Theorem}
\newtheorem{lemma}{Lemma}
\newtheorem{corollary}{Corollary}
\newtheorem{proposition}{Proposition}

\title{Dispersion dynamics for the defocusing generalized Korteweg-de Vries equation}
\author{Stefan Steinerberger}
\address{Mathematisches Institut, Universit\"at Bonn, Endenicher Allee 60, 53115 Bonn, Germany}

\begin{document}
\begin{abstract}
We study dispersion for the defocusing gKdV equation.
It is expected that it is not possible for the bulk of the $L^2-$mass to concentrate in a small interval for a long time.
We study a variance-type functional exploiting Tao's monotonicity formula in the spirit of earlier work by Tao as well as Kwon \& Shao
and quantify its growth in terms of sublevel estimates.
\end{abstract}

\maketitle

\section{Introduction}
We are interested in the dispersion properties of the defocusing generalized Korteweg de Vries (gKdV) equation. The gKdV equation
is given by
$$ \partial_t u + \partial_{xxx}u = \mu\partial_x(|u|^{p-1}u),$$
where $p>1$, $\mu$ is real ($\mu = 1$ corresponds to defocusing) and $u$ is real-valued. There are conservation laws for mass and energy
\begin{align*} M(u) &= \int_{\mathbb{R}}{u(t,x)^2dx}\\
E(u) &= \int_{\mathbb{R}}{\frac{1}{2}u_x(t,x)^2 + \frac{\mu}{p+1}|u(t,x)|^{p+1}dx}
\end{align*}
as well as a scaling symmetry
$$ u(t,x) \rightarrow \lambda^{-\frac{2}{p-1}}u\left(\frac{t}{\lambda^3}, \frac{x}{\lambda}\right).$$
The focusing case $\mu = 1$ has been intensively studied: we refer to the influential work of Martel \& Merle; in particular, their Liouville-type
theorem \cite{mm} states that (under some conditions) a perturbation of the soliton (away from the soliton manifold) cannot be $L^2-$compact.
Based on the extensive theory developed for the focusing gKdV, there are also results by de Bouard \& Martel \cite{deb} and Laurent \& Martel \cite{lau}.\\

The defocusing case $\mu = +1$ does not only \textit{not} have solitons, it is also expected to not have 'pseudosolitons' (solutions 
whose $L^2-$mass exhibits spatial concentration over time) in a fairly general sense. This has been fully resolved in the mass-critical 
case $p=5$ by Dodson \cite{dob2} (building on earlier work by Killip, Kwon, Shao \& Visan \cite{ki} and an older result \cite{dob} of his). 
We emphasize that we are interested in the long-time dynamical behavior of smooth initial data $u_0 \in \mathcal{S}(\mathbb{R})$: the defocusing 
nature of the equation guarantees global well-posedness, including uniform bounds on $\|u\|_{L^{\infty}_{t,x}}$ and $\|u\|_{L^{\infty}_t H^1_x}$
and regularity never becomes an issue. A similar problem occurs for the defocusing, one-dimensional nonlinear wave equation
$$ -u_{tt} + u_{xx} = |u|^{p-1}u \qquad \mbox{on}~\mathbb{R}$$
with $\|u(0)\|_{H^1(\mathbb{R})} + \|\partial_t u(0)\|_{L^2(\mathbb{R})} < \infty$. 
Here, again, $\|u\|_{L^{\infty}_{t,x}}$ bounds are easy and the question is whether there is an actual decay of $\|u(t)\|_{L^{\infty}_x}$.
In a certain averaged sense, this was proven by Lindblad \& Tao \cite{lind}.\\

Using weighted integrals of the conservation laws (an idea dating back at least to Friedrichs and famously used in the work of Morawetz \cite{mor}),
the first result for the defocusing gKdV is due to Tao \cite{tao}. He introduced the normalized centers of mass and energy via
\begin{align*}
 \left\langle x \right\rangle_M := \frac{1}{M}\int_{\mathbb{R}}{xu^2dx} \quad \mbox{and} \quad \left\langle x \right\rangle_E := \frac{1}{E}\int_{\mathbb{R}}{x\left(\frac{1}{2}u_x^2 + \frac{1}{p+1}|u|^{p+1}\right)dx},
\end{align*}
where $M$ and $E$ denote mass and energy of the solution, respectively. If we assume dynamics from the linear part to
play the dominant role in the nonlinear behavior, then high frequencies of the solutions should move to $-\infty$ much quicker than slow frequencies.
However, since energy is more weighted toward high frequencies, we could hope for some connection between the centers of mass
and energy. 

\begin{thm}[Tao] Let $p \geq \sqrt{3}$ and let $u$ be a global-in-time Schwartz solution. Then
 $$ \partial_{t} \left\langle x \right\rangle_E <  \partial_{t} \left\langle x \right\rangle_M.$$
\end{thm}

The condition $p \geq \sqrt{3}$ is required in the proof but might be an artifact. The statement itself (and all
subsequent statements based on it) are formal in nature, however, it is expected that they all remain true for rougher
solutions. The monotonicity implies the following dispersion statement.

\begin{thm}[Tao] Let $p \geq \sqrt{3}$ and let $u$ be a global-in-time Schwartz solution. Then, for any function $x:\mathbb{R} \rightarrow \mathbb{R}$,
$$ \sup_{t \in \mathbb{R}} \int_{\mathbb{R}}{|x-x(t)|\left(u(t,x)^2 + \frac{u_x(t,x)^2}{2} + \frac{|u(t,x)|^{p+1}}{p+1}\right)dx} = \infty.$$
\end{thm}

This has no implications on spatial decay of the solution because some energy might move far away very quickly. Kwon \& Shao \cite{ks} noticed
that the monotonicity formula could be employed in another way to exclude some type of spatial decay.

\begin{thm}[Kwon \& Shao] Let $p \geq \sqrt{3}$ and let $u$ be a global-in-time Schwartz solution. Then, for any function $x:\mathbb{R} \rightarrow \mathbb{R}$,
 $$ \sup_{t \in \mathbb{R}} \int_{\mathbb{R}}{(x-x(t))^2u(t,x)^2dx} = \infty.$$
\end{thm}
The key observation is the simple fact that
$$ \int_{\mathbb{R}}{(x-x(t))^2u(t,x)^2dx} \geq \int_{\mathbb{R}}{(x-\left\langle x \right\rangle_M)^2u(t,x)^2dx},$$
where the right hand side allows for an explicit computation. In particular, if a solution satisfies
$$ |u(t,x)| \lesssim \frac{1}{|x-x(t)|^{3/2+\varepsilon}}$$
for some $x(t): \mathbb{R} \rightarrow \mathbb{R}$, then $u \equiv 0.$ Inspired by these results, we study the interaction functional 
$I: L^2(\mathbb{R}) \times L^2(\mathbb{R}) \rightarrow \mathbb{R} \cup \left\{\infty\right\}$ given by
$$ I(f) := \int_{\mathbb{R}}{\int_{\mathbb{R}}{f(x)^2(x-y)^2f(y)^2dx}dy}.$$
Formal calculations with a translation invariant interaction term $\eta(x-y)$ yields complicated expressions of the form $\eta'''$,
which naturally suggests $\eta(x-y) = (x-y)^2$. Interaction estimates of this type have been used very effectively for NLS: we refer to
the work of Colliander \& Grillakis \& Tzirakis \cite{col} who use a weight that is essentially quadratic for small distances
and results by Planchon \& Vega \cite{pla} (where a similar computation is mentioned but they ultimately manage to use a lower 
order expression $\eta(x-y) = |(x-y)\cdot \omega|$, where $\omega$ is a fixed unit vector). Quadratic Morawetz estimates for NLS
and Hartree equations are also studied by Ginibre \& Velo \cite{gv}.\\

Our functional incorporates structure coming from Tao's theorem as follows: for any solution $u(t)$ of the gKdV, the 
variation of the functional in time can be written as
\begin{align*}
   \partial_t I(u(t)) = 12 E M\left(\left\langle x \right\rangle_M  -  \left\langle x \right\rangle_E \right) + \frac{4p-12}{p+1}\left(\left\langle x \right\rangle_M M \int_{\mathbb{R}}{|u|^{p+1}dx} - M \int_{\mathbb{R}}{|u|^{p+1}x dx}\right).
\end{align*} 
So far, this is very much in spirit of the earlier results by Tao and Kwon \& Shao and unboundedness of the functional follows from their approach. 
A key novelty in our approach is a slight refinement of Tao's monotonicity which yields additional control on $\partial_t I(u(t))$ in terms
of mass, energy and $I(u(t))$, which allows for bootstrapping and the derivation of some additional information such as sublevel estimates
on $I(u(t))$.

\section{Statement of Results}

 Our result relies on Tao's monotonicity and inherits
all its requirements. We believe the condition on $p$ to be an artifact and that all statements
hold true for much rougher solutions $u$ as well. 
\begin{thm}  Let $p \geq \sqrt{3}$ and let $u$ be a global-in-time Schwartz solution. Then
$$ \left|\left\{t>0: I(u(t)) \leq z \right\}\right| \lesssim_{u(0)} z^{\frac{p}{2}}.$$
\end{thm}
Our emphasis is on the following: the functional $I(u(t))$ is neither monotonically 
increasing nor convex (except for $p=3$) but can nonetheless only be small for a 
bounded amount of time. The proof is based on showing that $I(u(t))$ satisfies a certain integrodifferential
inequality, where the influence of nonlinear dynamics acting on regions of mass concentration can
be clearly observed. The precise dependence of the implicit constant on the initial data follows 
from the proof and is related to mass, energy and the centers of mass and energy.\\

The following minor improvement of the Kwon-Shao nonexistence result is a trivial consequence.
\begin{corollary} Let $\varepsilon > 0$ be fixed. Under the assumptions of Theorem 1, if 
$$ |u(t,x)| \lesssim \frac{(1+t)^{\frac{2}{p}-\varepsilon}}{|x-x(t)|^{\frac{3}{2}+\varepsilon}}$$
for some $x:\mathbb{R}\rightarrow \mathbb{R}$, then $u\equiv 0$.
\end{corollary}
For the particular case of the mKdV ($p=3$), some algebraic simplifications immediately
imply that the functional $I(u(t))$ is convex. However, in this case the connection to the focusing KdV
via the Miura map should give a wealth of additional information anyway -- the problem can be
expected to be much simpler in this special case. \\

We consider the following corollary to be a much more interesting consequence: its conditions
on center and mass are immediately seen to be necessary for the argument to work, however, 
heuristic arguments (to be described below) make them also seem necessary.

\begin{corollary}
  Let $p \geq \sqrt{3}$ and let $u$ be a global-in-time Schwartz solution. There exists a constant
$c>0$ depending only on $p$ such that if
 $$  \left\langle x \right\rangle_E\big|_{t=0} \leq  \left\langle x \right\rangle_M\big|_{t=0},$$
and
$$ I(u(0)) \leq c M^{4-\frac{p+3}{p-1}}E^{-1}$$
then
$$ \inf_{t > 0}{I(u(t))} \geq \frac{1}{4}I(u(0)).$$
\end{corollary}
\textit{Remark.} As we will show below, the assumption $I(u(0)) \leq c M^{4-\frac{p+3}{p-1}}E^{-1}$ does imply $M < c'$
for some constant $c' = c'(c,p)$. The entire statement is thus only applicable to initial data with small mass.\\

\textit{Interpretation.} If the function is sufficiently localized and very smooth, then at least
some part of it needs to break off and go away quickly never to return. The assumption on the centers
is not at all unreasonable: assume $u(0)$ is some $L^2-$normalized bump function localized in space
$x \sim 0$ and Fourier space $\xi \sim 1$ and add a small perturbation $w$ localized around
$x \sim x_0 \gg 1$ and $\xi \sim N$. Then, for $x_0 \gg \|w\|_{L^2}^{-1}$,
$$ I(u(0)) \sim x_0^2 \|w\|_{L^2}^2.$$
Assuming $w$ to be small in $L^{\infty}$, we expect linear dynamics to be dominating. This 
means that the perturbation $w$ moves with speed $-3N^2$ and noticeably decreases the functional (if
$N \gg \sqrt{x_0}$) while the big bump function $u_0$ barely moves at all during that time. We need to
exclude this scenario and indeed, for centers of mass and energy, we have
$$  \left\langle x \right\rangle_M\big|_{t=0} \sim x_0 \|w\|_{L^2}^2 \qquad 
 \left\langle x \right\rangle_E\big|_{t=0} \sim x_0 \|w_x\|_{L^2}^2 \sim x_0 N^2 \|w\|_{L^2}^2.$$
The condition $ \left\langle x \right\rangle_E\big|_{t=0} \leq  \left\langle x \right\rangle_M\big|_{t=0}$
now implies $N \lesssim 1$ but in that case the perturbation actually moves slower than $u(0)$ and the
problem cannot occur.\\

A small caveat: we need to make sure that 
$$ \mbox{no inequality of the type} \quad I(u(0)) \geq c M^{4-\frac{p+3}{p-1}}E^{-1} \quad \mbox{holds.}$$
Otherwise the statement would be a statement about the energy landscape of the functional $I$ and not about
the dynamics of the equation. We give a quick classification of all $(\alpha, \beta) \in \mathbb{R}^2$
for which $I(u) \gtrsim M(u)^{\alpha}E(u)^{\beta}$ holds true.

\begin{proposition} Let $p > 1$,
 $$3 \leq \alpha \leq \frac{4p}{p-1} \qquad \mbox{and} \qquad \beta = \frac{(4-\alpha)p+5\alpha - 12}{p+3}.$$
Then there exists a constant $c_p > 0$ such that for any function $u \in H^1(\mathbb{R})$ 
$$ \left(\int_{\mathbb{R}}{\frac{u_x^2}{2} + \frac{|u|^{p+1}}{p+1}dx}\right)^{\beta}\left[ \int_{\mathbb{R}}{\int_{\mathbb{R}}{u(x)^2(x-y)^2u(y)^2dx dy}} \right] \geq c_p\left(\int_{\mathbb{R}}{u^2dx}\right)^{\alpha}.$$
\end{proposition}
A simple scaling argument shows that these are the only $(\alpha, \beta)$ for which such an inequality
can possibly hold. In particular, setting $\alpha = 3$ gives $\beta = -1$ and therefore
$$ I(0) \gtrsim M^3 E^{-1},$$
which renders Corollary 2 nontrivial but also shows that it is only applicable in the case of small mass. 
These inequalities seem fairly technical and of little intrinsic interest. That is why we were
surprised about the following: Proposition 1 is implied by Sobolev embedding and interpolation with 
the following elementary inequality, which we couldn't find in the literature. In an earlier version
of the manuscript, we proved existence of extremizers and gave some rough bounds on
$c_p$. However, these extremizers can be found explicitely via the Lagrange multiplier theorem (an 
observation that was communicated to us by Soonsik Kwon).

\begin{proposition}
 Let $p > 1$ and
$$ c_p = \frac{1}{2\pi}\frac{\Gamma\left(\frac{p+1}{p-1}\right)}{\Gamma\left(\frac{2p}{p-1}\right)}
\left(\frac{\Gamma\left(\frac{2p}{p-1}\right)}{\Gamma\left(\frac{5p-1}{2p-2}\right)}\right)^{\frac{p+3}{p-1}}
\left(\frac{\Gamma\left(\frac{3p+1}{2p-2}\right)}{\Gamma\left(\frac{p+1}{p-1}\right)}\right)^{\frac{3p+1}{p-1}} \geq \frac{1}{2\pi e}.$$
 Then, for every $u \in L^2(\mathbb{R})$,
$$\left(\int_{\mathbb{R}}{u^2x^2 dx}\right)\left(\int_{\mathbb{R}}{ |u|^{p+1}dx}\right)^{\frac{4}{p-1}} \geq c_p\left(\int_{\mathbb{R}}{u^2dx}\right)^{\frac{3p+1}{p-1}}.$$
Additionally, all minimizers are given by translation, scaling and dilation of the compactly supported function
$$ u(x) = \begin{cases}
           (1-x^2)^{\frac{1}{p-1}} \qquad &\mbox{if}~|x| \leq 1\\
0 \qquad &\mbox{otherwise.}
          \end{cases}$$
\end{proposition}

 \textit{Remark.} Suppose $u \in H^1(\mathbb{R})$. Then there is the classical uncertainty principle in the form
$$ \|u x\|_ {L^2}\|u_x\|_ {L^2} \geq \frac{1}{2}\|u\|_{L^2}^2.$$
Combining our inequality with the Gagliardo-Nirenberg inequality
$$ \|u\|_{L^{p+1}}^{p+1} \leq G_p \| u\|_{L^2}^{\frac{p+3}{2}}\|u_x\|_{L^2}^{\frac{p-1}{2}}$$
yields a version of the uncertainty principle with different constants and an additional term sandwiched in the middle
$$ \|u x\|_ {L^2}\|u_x\|_ {L^2} \geq \frac{1}{G_p^{\frac{2}{p-1}}} \frac{\|u x\|_{L^2} \|u\|_{L^{p+1}}^{\frac{2p+2}{p-1}}}{\|u\|_{L^2}^{\frac{p+3}{p-1}}} \geq \frac{\sqrt{c_p}}{G_p^{\frac{2}{p-1}}}\|u\|_{L^2}^2.$$

\section{Proof of Theorem 1}
We start by giving a refined version of Tao's monotonicity formula, then
derive a functional inequality and conclude by showing the sublevel estimate for all solutions of the differential inequality. Throughout the paper
$$ I(u(t)) := \int_{\mathbb{R}}{\int_{\mathbb{R}}{u(t,x)^2 (x-y)^2 u(t,y)^2dxdy}}.$$

\subsection{Refined monotonicity formula.}
The statement is implicitely contained in Tao's proof.
\begin{lemma}[Refined monotonicity formula] Let $u(t,x)$ be a global-in-time Schwartz solution to the defocusing gKdV for some $p \geq \sqrt{3}$. Then
$$\partial_t\left\langle x \right\rangle_M  - \partial_t\left\langle x \right\rangle_E \gtrsim_p \frac{1}{E M^3} \left(\int_{\mathbb{R}}{|u(x)|^{p+1}dx}\right)^2.$$
\end{lemma}
\begin{proof} Tao's proof is based on introducing
strictly positive quantities $a,b,q,r,s$ by solving
\begin{align*}
  a^2M &= \int_{\mathbb{R}}{u_{xx}^2dx} \qquad  \qquad b^2M = \int_{\mathbb{R}}{|u|^{2p}dx} \qquad \qquad aqM = \int_{\mathbb{R}}{u_x^2 dx} \\
brM &= \int_{\mathbb{R}}{|u|^{p+1}dx} \qquad \qquad absM = p\int_{\mathbb{R}}{|u|^{p-1}u_x^2 dx},
\end{align*}
where partial integration implies
$$ 0 < q,r,s < 1 \qquad \mbox{and} \qquad 1-q^2-r^2-s^2+2qrs \geq 0.$$
Then the proof can be finished algebraically by showing that for $p \geq \sqrt{3}$
$$EM(\partial_t\left\langle x \right\rangle_M  - \partial_t\left\langle x \right\rangle_E) = \frac{3}{2}(1-q^2)a^2 + \left(2s-\frac{p+3}{p+1}qr\right)ab + \frac{1}{2}\left(1-\frac{4p}{(p+1)^2}r^2\right)b^2 > 0.$$
However, he actually proves the stronger statement
$$\frac{3}{2}(1-q^2)a^2 + \left(2s-\frac{p+3}{p+1}qr\right)ab + \frac{1}{2}\left(1-r^2\right)b^2 > 0.$$
Hence
$$ EM(\partial_t\left\langle x \right\rangle_M  - \partial_t\left\langle x \right\rangle_E) \geq \left(\frac{1}{2} - \frac{2p}{(p+1)^2}\right)b^2 r^2 \gtrsim_p \frac{1}{M^2}\left(\int_{\mathbb{R}}{|u(x)|^{p+1}dx}\right)^2.$$
\end{proof}

\subsection{A differential inequality.} Repeated partial integration yields
\begin{align*}
 \qquad \partial_t \int_{\mathbb{R}}{\int_{\mathbb{R}}{u(x)^2u(y)^2(x-y)^2dxdy}}  =  4\int_{\mathbb{R}}{\int_{\mathbb{R}}{u(y)^2\left(\frac{p}{p+1}|u(x)|^{p+1} + \frac{3}{2}u_x(x)^2\right)\left(y-x\right)dx dy}}.
\end{align*}
 
By introducing the normalized centers of mass and energy,
$$\left\langle x \right\rangle_M := \frac{1}{M}\int_{\mathbb{R}}{xu^2dx}  \qquad \mbox{and} \qquad \left\langle x \right\rangle_E := \frac{1}{E}\int_{\mathbb{R}}{x\left(\frac{1}{2}u_x^2 + \frac{1}{p+1}|u|^{p+1}\right)dx},$$
we can rewrite the first derivative as
\begin{align*}
  (\diamondsuit) \qquad \partial_t\int_{\mathbb{R}}{\int_{\mathbb{R}}{u(x)^2u(y)^2(x-y)^2dxdy}} &= 12 E M\left(\left\langle x \right\rangle_M  -  \left\langle x \right\rangle_E \right) \\
 &+ \frac{4p-12}{p+1}\left(\left\langle x \right\rangle_M M \int_{\mathbb{R}}{|u|^{p+1}dx} - M \int_{\mathbb{R}}{|u|^{p+1}x dx}\right).
\end{align*} 
In the special case $p=3$, the monotonicity formula immediately implies convexity of $I(u(t))$. \\

The following simple argument will be used also in later proofs. If $I(u(t))$ is small, then there is a small interval containing a lot of the $L^2-$mass: for a fixed time $t$, let $J$ be the unique interval such that
$$ \int_{x < \inf J}{u^2 dx} = \int_{x > \sup J}{u^2 dx} = \frac{1}{4}\int_{\mathbb{R}}{u^2 dx},$$
then
    $$ \frac{M^2}{4}|J|^2 \leq \int_{x \in \mathbb{R} \setminus J}{\int_{x \in \mathbb{R} \setminus J}{ u(t,x)^2 u(t,y)^2(x-y)^2dxdy}} \leq I(u(t)),$$
which implies $|J| \leq 4\sqrt{I(u(t))}/M.$ Hence, with H\"older,
$$ \frac{M}{2} = \int_{J}{u^2 dx} \leq \left(\int_{J}{|u|^{p+1} dx}\right)^{\frac{2}{p+1}}|J|^{\frac{p-1}{p+1}}$$
and thus, as a consequence,
$$  \left(\int_{\mathbb{R}}{|u|^{p+1} dx}\right)^{2} \geq \left(\int_{J}{|u|^{p+1} dx}\right)^{2} \gtrsim M^{2p}I(u(t))^{1/2-p/2}.$$
The fundamental theorem of calculus and refined monotonicity (i.e. Lemma 1) give 
\begin{align*}
 (\left\langle x \right\rangle_M  - \left\langle x \right\rangle_E)\big|_{t} -  (\left\langle x \right\rangle_M  - \left\langle x \right\rangle_E)\big|_{t=0} 
&\gtrsim \int_{0}^{t}{\frac{1}{E M^3}\left(\int_{\mathbb{R}}{|u(z,x)|^{p+1}dx}\right)^2dz} \\
&\gtrsim \frac{M^{2p-3}}{E}\int_{0}^{t}{I(u(z))^{\frac{1-p}{2}}dz}.
\end{align*}

The remaining term on the right-hand side of $(\diamondsuit)$ can be easily controlled via
$$ \left|\left\langle x \right\rangle_M M \int_{\mathbb{R}}{|u|^{p+1}dx} - M \int_{\mathbb{R}}{|u|^{p+1}x dx}\right| \leq  \int_{\mathbb{R}}{\int_{\mathbb{R}}{u(x)^2|u(y)|^{p+1}|x-y|dxdy}},$$
which, using H\"older and
$$ \|u\|_{L^{\infty}} \lesssim \left(\int_{\mathbb{R}}{u^2dx}\right)^{\frac{1}{4}}\left(\int_{\mathbb{R}}{u_x^2dx}\right)^{\frac{1}{4}} \lesssim M^{1/4}E^{1/4}$$
can be bounded by
$$  \int_{\mathbb{R}}{\int_{\mathbb{R}}{u(x)^2|u(y)|^{p+1}|x-y|dxdy}} \lesssim M(ME)^{\frac{p-1}{4}}I(u(t))^{1/2}.$$

Altogether, this yields that the real function $f(t) := I(u(t))$ satisfies the differential inequality
$$ f'(t) \geq \alpha \int_{0}^{t}{f(z)^{\frac{1-p}{2}}dz} - \beta\sqrt{f(t)} - \gamma,$$
for positive constants $\alpha \sim M^{2p-2}$, $\beta \sim M(EM)^{\frac{p-1}{4}}$ and a constant $\gamma$ which encodes the initial difference between
the center of mass and the center of energy
$$ \gamma = E M (\left\langle x \right\rangle_M  - \left\langle x \right\rangle_E)\big|_{t=0}.$$

\subsection{Conclusion.} This section finishes the proof of Theorem 1 by showing that the derived differential inequality alone already implies the sublevel estimate.
\begin{lemma} Let $f \in C^1(\mathbb{R}, \mathbb{R}_+)$, $\alpha, \beta > 0$ and $\gamma \in \mathbb{R}$ arbitrary. Assume that $f(x) > \delta > 0$ and
$$ f'(x) \geq \alpha \int_{0}^{x}{f(y)^{\frac{1-p}{2}}dy} - \beta\sqrt{f(x)} - \gamma,$$
then 
$$ \mu\left( \left\{x>0: f(x) \leq z\right\}\right) \lesssim_{\alpha, \beta, \gamma, \delta} z^{\frac{p}{2}}.$$ 
\end{lemma}
\begin{proof} Let us quickly describe the argument: the lower bound is comprised of one trivial component $- \beta\sqrt{f(x)} - \gamma$, which merely depends on the
actual value of $f(x)$ and one term with 'memory': whether the integral is large compared to the trivial component depends on whether or not the function has been small in the past. 
In particular, if the function has been small in the past for a long time, the integral will dominate the trivial component and force the function to grow. 
Fix a $z > \delta$, let $c$ to be a large positive constant, consider $I = \left\{x>0: f(x) \leq z\right\}$ and take $K$ sufficiently large such that
$$|[0,K] \cap I| = cz^{\frac{p}{2}}.$$
It remains to show that taking $c>0$ sufficiently large yields a contradiction. Let us take
a look at the derivative at $x = K$ (trivially, $K$ can be chosen such that $K \in I$)
 \begin{align*}
   f'(K) &\geq \alpha \int_{0}^{K}{f(y)^{\frac{1-p}{2}}dy} - \beta\sqrt{f(K)} - \gamma \geq \alpha \int_{I \cap [0,K]}^{}{f(y)^{\frac{1-p}{2}}dy} - \beta\sqrt{z} - \gamma \geq \alpha c \sqrt{z}  - \beta\sqrt{z} - \gamma
 \end{align*}
Since $f(x) > \delta > 0$, the statement we are trying to prove is trivially true for $z \leq \delta$ and we may assume $z \geq \delta$. Then, 
for any $c$ suffiently large depending on $\alpha, \beta, \gamma$, this implies that at 
$x = K$ the derivative is of order $f'(x) \gtrsim \sqrt{z}$. By the
same reasoning, the same holds true for all points in $(K, \infty) \cap I$. This growth  implies that $(\mathbb{R} \setminus I) \cap (K, K + C\sqrt{z}) \neq \emptyset,$
where the constant $C$ depends on $\alpha, \beta, \gamma$ but not $z$. We show now that $\sup I \leq K + C\sqrt{z}$. Suppose, this were not the case.
Since $(\mathbb{R} \setminus I) \cap (K, K + C\sqrt{z}) \neq \emptyset$, there would then be a smallest point $y^*$ with $f(y^*) = z$, where
the previous inequality implies $f'(y^*) > 0$ and this is a contradiction. Altogether, the final ingredient $f(x) > \delta > 0$ implies
$$ |I| \leq cz^{\frac{p}{2}} + C\sqrt{z} \lesssim_{\delta} (c+C)cz^{\frac{p}{2}},$$
where $c$ could be chosen depending only on $\alpha, \beta, \gamma$ and $C$ was finite.
\end{proof}
\textit{Remark.} Note that in applying this result to our case, the constants behave as $\alpha \sim M^{2p-2}$, $\beta \sim M(EM)^{\frac{p-1}{4}}$ and 
$$ \gamma = E M (\left\langle x \right\rangle_M  - \left\langle x \right\rangle_E)\big|_{t=0}.$$
For the constant $\delta > 0$, we have $\delta \gtrsim M^3 E^{-1}$ (which is proven further below), hence $\delta = \delta(\alpha, \beta)$, which is
why it is not mentioned in the formulation of the Theorem 1.

\section{Proof of Corollary 2}
This section first points out in which way the previous argument can be strenghtened to yield (very minor) additional information in a special case.
Its remainder is then devoted to the connections between the functional, energy and mass.

\subsection{Proof of Corollary 2.} The argument is very easy and uses nothing but the differential inequality.
\begin{lemma} Let $f \in C^1(\mathbb{R}, \mathbb{R}_+)$ and $\alpha, \beta > 0, \gamma \leq 0$. If
$$ f'(x) \geq \alpha \int_{0}^{x}{f(y)^{\frac{1-p}{2}}dy} - \beta\sqrt{f(x)} - \gamma$$
and
$$ f(0) < \left(\frac{\alpha}{\beta^2}\right)^{\frac{2}{p-1}},$$
then
$$ \inf_{x > 0}{f(x)} \geq \frac{f(0)}{4}.$$
\end{lemma}
\begin{proof} It follows from the structure of the inequality that we can restrict ourselves to functions
being monotonically increasing until they reach a global minimum.
The function
$$ g(x) = \frac{1}{4}(2\sqrt{f(0)} - \beta x)^2$$
satisfies $g(0) = f(0)$ and $ g'(x) = -\beta \sqrt{g(x)}$. At $x = \sqrt{f(0)}/\beta$, we have
from monotonicity and the initial bound on $f(0)$ that
$$ f'(x) \geq \alpha \int_{0}^{x}{f(y)^{\frac{1-p}{2}}dy} - \beta\sqrt{f(x)} \geq \frac{\alpha}{\beta}\sqrt{f(0)}f(0)^{\frac{1-p}{2}} - \beta\sqrt{f(0)} > 0$$
implying that the minimum is assumed before $x = \sqrt{f(0)}/\beta$. However, 
$$ g\left(\frac{\sqrt{f(0)}}{\beta}\right) =  \frac{1}{4}(2\sqrt{f(0)} - \beta \frac{\sqrt{f(0)}}{\beta})^2 = \frac{f(0)}{4}.$$
\end{proof}

\subsection{Reducing Proposition 1 to Proposition 2.} The product structure of the inequalities implies that it is sufficient to prove the two endpoints
$\alpha = 3$ and $\alpha = 4p/(p-1)$. The endpoint $\alpha = 3$ follows quickly from Sobolev embedding. All further considerations are for general functions 
$u \in H^1(\mathbb{R})$, where we use $M$ and $E$ to denote their mass and energy, respectively. There are no time-dependent elements in the arguments nor
does the gKdV equation play any role. 

\begin{lemma} For any $u \in H^1(\mathbb{R})$
 $$ I(u) \gtrsim \frac{M^{3}}{E^{1}}.$$
\end{lemma}
\begin{proof}
By Sobolev embedding
$$ \|u\|^2_{L^{\infty}} \leq \left(\int_{\mathbb{R}}{u^2dx}\right)^{\frac{1}{2}}\left(\int_{\mathbb{R}}{u_x^2dx}\right)^{\frac{1}{2}} \lesssim M^{1/2}E^{1/2}.$$
At the same time, reusing an argument employed earlier, there is an interval $J$ of length  $|J| \lesssim \sqrt{I(u)}/M$ such that $J$ contains half of the $L^2-$mass of $u$. Therefore
$$ \|u\|^2_{L^{\infty}} \geq \|u\|^2_{L^{\infty}(J)} \geq \frac{1}{|J|}\int_{J}{u(x)^2dx} \gtrsim \frac{M}{|J|} \gtrsim \frac{M^2}{\sqrt{I(u)}}.$$
Combining these two inequalities gives the result.
\end{proof}

The proof of the second endpoint uses symmetric decreasing rearrangement to gain an additional symmetry, which then yields an algebraic simplification of the functional.
The following statement will come as no surprise at all, it can certainly be founded in the literature in more general form.

\begin{lemma}
 The functional $I$ is decreasing under symmetrically decreasing rearrangement.
\end{lemma}
\begin{proof}
We use
the familiar layer-cake decomposition
$$ \int_{\mathbb{R}}{\int_{\mathbb{R}}{u(t,x)^2 u(t,y)^2(x-y)^2dxdy}} = \int_{0}^{\infty}{\int_{0}^{\infty}{r s \int_{\left\{u(x)^2=r\right\} \times \left\{u(y)^2=s\right\}}{(x-y)^2d\mathcal{H}^2}dr}ds},$$
where $\mathcal{H}^2$ is the $2-$dimensional Hausdorff measure. The statement would then follow if it were the case that for fixed positive $a,b > 0$ and
subsets $A, B \subset \mathbb{R}$ of corresponding size 
$$ \inf_{|A| = a, |B| = b} \int_{A}{\int_{B}{(x-y)^2 dx} dy}$$
is assumed precisely when $A, B$ are intervals with the same midpoint (potentially ignoring Lebesgue null sets in the process). Let $A, B$ be
hypothetical counterexamples, then there exist constants $c_1, c_2$ such that a neighbourhood of $c_1$ is not contained in $A$ and both
$A_1 := A \cap \left\{x: x > c_1\right\}$ and $A_2 := A \cap \left\{x: x< c_1\right\}$ are nonempty and likewise for $c_2$ and $B$. Let us then replace
$A_1$ and $B_1$ by $A_1 + \varepsilon$ and $B_1 + \varepsilon$ for sufficiently small $\varepsilon$ such that no overlap occurs. Then the
integration between $A_1$ and $B_1$ as well as between $A_2$ and $B_2$ remains unchanged while it decreases between $A_1$ and $B_2$ as well
as $A_2$ and $B_1$. This shows that the infimum can only be assumed a pair of intervals (up to Lebesgue null sets) and an explicit calculation
yields the midpoint property.
\end{proof}

For symmetric functions, the functional simplifies to
$$ \int_{\mathbb{R}}{\int_{\mathbb{R}}{u(x)^2(x-y)^2u(y)^2dxdy}} = 2\left(\int_{\mathbb{R}}{u(x)^2 x^2 dx}\right)\left(\int_{\mathbb{R}}{u(x)^2dx}\right).$$
Assuming a lower bound of the type $I \gtrsim M^{\alpha}E^{\beta}$, we may use the fact that the gKdV scaling acts nicely on $M$ and $E$ to derive
the necessary condition
$$ 4 - \frac{8}{p-1} = \alpha\left(1-\frac{4}{p-1}\right) + \beta\left(3-\frac{4p}{p-1}\right).$$
A standard scaling $u(\cdot) \rightarrow a u(b \cdot)$ with $a,b > 0$ implies $2\alpha + (p+1)\beta \leq 4$ and
$2 \leq \alpha + \beta \leq 4$. These scaling considerations have the endpoint $(\alpha, \beta) = (3,-1)$, which is implied by the previous Lemma (where we actually didn't need the full energy but
only $\|u_x\|_{L^2}^2$). The other endpoint is simply Proposition 2.

\subsection{Proof of Proposition 2.} Here we give a proof of Proposition 2. The argument first shows
$$\forall p>1~ \exists c_p>0~ \forall u \in L^2(\mathbb{R}) \qquad \left(\int_{\mathbb{R}}{u^2x^2 dx}\right)\left(\int_{\mathbb{R}}{ |u|^{p+1}dx}\right)^{\frac{4}{p-1}} \geq c_p\left(\int_{\mathbb{R}}{u^2dx}\right)^{\frac{3p+1}{p-1}}.$$
It is based on using invariance under scaling and dilations and a characterization of compactness in $L^2$. It implies the existence of a minimizer but no numerical bounds on $c_p$.
The second step, based on an observation of Soonsik Kwon, is that the Euler-Lagrange functional has a closed form solution.
\begin{proof} It suffices to consider functions $u$ invariant under symmetric decreasing rearrangement.
We pick a minimizing sequence $u_n \in H^1(\mathbb{R})$ of
$$ \frac{\left(\int_{\mathbb{R}}{u^2x^2 dx}\right)^{\frac{p-1}{4}} \int_{\mathbb{R}}{ u^{p+1}dx}}{\left(\int_{\mathbb{R}}{u^2dx}\right)^{\frac{3p+1}{p-1}}}$$
and use invariance under scaling and dilations to prescribe $u(0) = 1$ and $\int{u^2 x^2 } = 1$. Trivially, the sequence is then bounded
in $L^2$ by
$$ \int_{\mathbb{R}}{u^2 dx} \leq 2 + \int_{|x| \geq 1}{u^2 x^2 dx} = 3.$$
We use an observation that is usually ascribed to either Feichtinger \cite{fei} or Pego \cite{pego}. A bounded set $K \subset L^2$ is compact in $L^2$ if there exists a function
$C:\mathbb{R}_+ \rightarrow \mathbb{R}_+$ giving uniform control on decay
$$ \forall f \in K \qquad \int_{|x| > C(\varepsilon)}{|f(x)|^2 dx} + \int_{|\xi| > C(\varepsilon)}{|\hat{f}(\xi)|^2 d\xi} < \varepsilon.$$
The first part is easy and follows from the normalization. The symmetry of our functions implies that their Fourier transform is real-valued and thus
\begin{align*}
 \left|\hat{f}(\xi)\right| &= \left| \int_{-\infty}^{\infty}{f(x)\cos{\xi x}dx}\right| \lesssim \frac{1}{|\xi|},
\end{align*}
due to the fact that the monotonicity gives rise to an alternating series. Hence the minimizing sequence is compact and there exists
a minimizer. We fix 
$$ \int_{\mathbb{R}}{u^2 x^2 dx} = 1 = \int_{\mathbb{R}}{u^2}$$
and want to minimize
$$ \int_{\mathbb{R}}{|u|^{p+1} dx}$$
under these constraints. The Lagrange multiplier theorem implies that
$$ u^p = \lambda_1 x^2 u + \lambda_2 u$$
for some constants $\lambda_1, \lambda_2$ from which the statement follows.
\end{proof}

\end{document}